\documentclass[a4paper, 11pt]{amsart}

\usepackage{amsmath,amsthm,amssymb}
\usepackage{mathtools}
\usepackage{color}
\usepackage{enumerate}
\usepackage{ulem}
\usepackage[driverfallback=dvipdfm]{hyperref}
\usepackage[dvipdfmx]{graphicx}
\usepackage{tikz}
\usetikzlibrary{intersections, calc, arrows.meta}
\usepackage{bbold}
\usepackage{cite}

\theoremstyle{plain}
	\newtheorem{thm}{Theorem}[section]
	\newtheorem{cor}[thm]{Corollary}
	\newtheorem{lem}[thm]{Lemma}
	\newtheorem{prop}[thm]{Proposition}
	
	\newtheorem*{prob}{Problem}

\theoremstyle{definition}
	\newtheorem{dfn}[thm]{Definition}

\theoremstyle{remark}

\makeatletter
	
	\@addtoreset{equation}{section}
\makeatother

	\def\R{\mathbb{R}}
	\def\C{\mathbb{C}}
	
	\def\D{\Delta}

	\def\DL{\mathcal{D}}
	
	\def\IL{\mathcal{I}}
	
	\def\M{\mathcal{M}}
	
	\def\a{\alpha}
	\def\b{\beta}

	\def\cc{\overline}
	\newcommand\id{\mathrm{id}}
	\newcommand\0{\mathbb{0}}
	\newcommand\1{\mathbb{1}}
	\newcommand\eq{\Longleftrightarrow}
	
	\newcommand\CC{C[0,1]}
	\newcommand\CCR{C_\R[0,1]}
	\newcommand\CD{C^1[0, 1]}
	\newcommand\CR[1]{C_\R(#1)}
	\newcommand\CDR{C^1_\R[0, 1]}
	\newcommand\Lip{\operatorname{Lip}[0, 1]}
	\newcommand\LipR{\operatorname{Lip}_\R[0, 1]}
	\newcommand\AC{\operatorname{AC}[0, 1]}
	\newcommand\ACR{\operatorname{AC}_\R[0, 1]}
	\newcommand\PBCR[1]{B^+(C_\R(#1))}
	\newcommand\PSCR[1]{S^+(C_\R(#1))}

	\newcommand\PSsp{S_{C^1, p}^+}

	\newcommand\PSLsp{S_{\operatorname{Lip}, p}^+}

	\newcommand\PSREp[1]{S^+(\R \oplus_p E_{#1})}
	\newcommand\PSRCp[1]{S^+(\R \oplus_p C_\R(#1))}
	\newcommand\PSRCm[1]{S^+(\R \oplus_\infty C_\R(#1))}
	\newcommand\PSCCp[1]{S^+(\C \oplus_p C(#1))}

	\newcommand\PBE[1]{B^+(E_{#1})}

	\newcommand\norm[1]{\|#1\|}
	
	\newcommand\sn[1]{\left\|#1\right\|_\infty}
	\newcommand\Cp[1]{\left\|#1\right\|_{C, p}}
	\newcommand\Sigp[1]{\left\|#1\right\|_{\Sigma, p}}
	\newcommand\Sigm[1]{\left\|#1\right\|_{\Sigma, \infty}}
	\newcommand\sigp[1]{\left\|#1\right\|_{\sigma, p}}
	\newcommand\sigm[1]{\left\|#1\right\|_{\sigma, \infty}}
	\newcommand\Mnorm[1]{\left\|#1\right\|_M}
	\newcommand\mnorm[1]{\left\|#1\right\|_m}
	\newcommand\normp[1]{\left\|#1\right\|_p}
	\newcommand\normm[1]{\left\|#1\right\|_\infty}

	\newcommand\set[1]{\left\{\,#1\,\right\}}

	\let\diam\relax
	\DeclareMathOperator{\diam}{diam}

\title[Isometries on $S_{C^1[0,1]}^+$]{Surjective isometries on the positive parts of the unit spheres of some function spaces}

\author[Y. Enami]{Yuta Enami}
\address[Y. Enami]
{Graduate School of Science and Technology, Niigata University, Niigata 950-2181 Japan}
\email{f21j008j@mail.cc.niigata-u.ac.jp}

\author[D. Hirota]{Daisuke Hirota}
\address[D. Hirota]
{National Institute of Technology, Tsuruoka College
104 Sawada, Inooka, Tsuruoka, Yamagata 997-8511, Japan}
\email{dhirota@tsuruoka-nct.ac.jp}

\author[H. Koshimizu]{Hironao Koshimizu}
\address[H. Koshimizu]
{National Institute of Technology, Yonago College, Yonago 683-8502, Japan}
\email{koshimizu@yonago-k.ac.jp}

\author[M. R. Lin]{Min-Ruei Lin}
\address[M. R. Lin]
{Department of Applied Mathematics,
National Sun Yat-sen University,
Kaohsiung, 80424, Taiwan;
And
Graduate School of Science and Technology, Niigata University, Niigata 950-2181, Japan}
\email{m082030021@student.nsysu.edu.tw}


\begin{document}

\maketitle

\begin{abstract}
We consider the space $C^1[0, 1]$ of continuously differentiable functions on the closed unit interval $[0, 1]$ and the space $\operatorname{Lip}[0, 1]$ of Lipschitz continuous functions on $[0, 1]$, equipped with the norms 
\begin{align*}
\|f\|_{\sigma, p} = 
\begin{cases}
\sqrt[p]{|f(0)|^p + \|f'\|_\infty^p} & (1 \le p < \infty), \\
\max\{\, |f(0)|, \|f'\|_\infty \,\} & (p = \infty). 
\end{cases}
\end{align*}
We show that every surjective isometry on the positive part of the unit sphere extends to a surjective complex-linear isometry on the entire space. 
As a corollary, every such isometry also extends to an isometric order isomorphism on the real subspaces $C^1_{\mathbb{R}}[0, 1]$ and $\operatorname{Lip}_{\mathbb{R}}[0, 1]$. 
\end{abstract}


\section{Introduction}

Let $E$ and $F$ be Banach spaces, and let $A$ and $B$ be subsets of $E$ and $F$, respectively. 
A mapping $\Delta \colon A \to B$ is called an \textit{isometry} if 
\begin{align*}
\|\D(f) - \D(g)\| = \|f - g\| 
\end{align*}
for all $f, g \in A$. 

In 1972, Mankiewicz \cite{Man} proved that every surjective isometry between convex bodies in real Banach spaces can be extended to a surjective affine isometry between entire spaces. 
Subsequently, in 1987, Tingley \cite{Tin} posed the following question: 
Given normed spaces $E$ and $F$, does every surjective isometry between their unit spheres $S(E)$ and $S(F)$ extend to a surjective real-linear isometry between $E$ and $F$? 
This question is now known as \textit{Tingley's problem} and has motivated a substantial amount of research. 
Among the studies on function spaces, we refer the reader to \cite{CabCueEnaMiuPer, CueHirMiuPer, CuePer, Hat1, Hat2, HatOiShi, HirMiu, Wan1, Wan2}. 
Although many affirmative results have been obtained in various Banach spaces, the problem remains open in full generality. 

More recently, Peralta \cite{Per1} proposed several variants of Tingley's problem in the context of operator algebras. 
In particular, Peralta \cite{Per2} considered the question of whether every surjective isometry between the sets of positive norm-one elements in two $C^\ast$-algebras can be extended to a surjective complex-linear isometry between the $C^\ast$-algebras themselves. 
He obtained affirmative answers in the cases of the $C^\ast$-algebra of all bounded operators on Hilbert space and the $C^\ast$-algebra of all compact operators on separable Hilbert space. 

For an ordered Banach space $E$, we define the \textit{positive part of the unit sphere} of $E$ as
\begin{align*}
S^+(E) = \set{f \in E \colon \|f\| = 1, \,\, f \ge 0}. 
\end{align*}
Motivated by Peralta's question, Leung, Ng and Wong \cite{LeuNgWon} posed the following problem: 

\begin{prob}
Given ordered Banach spaces $E$ and $F$, is it true that every surjective isometry $\Delta \colon S^+(E) \to S^+(F)$ can be extended to a real-linear isometry $T \colon E \to F$?  
\end{prob}

We briefly summarize some known results concerning this problem in the context of function spaces. 
Leung, Ng and Wong \cite{LeuNgWon} solved this problem affirmatively for the Banach space $C(K)$ of continuous functions on a compact Hausdorff space $K$, as well as for $L^p$-spaces with $1 \le p \le \infty$. 
Furthermore, Ezumi, Lin and Miura \cite{EzuLinMiu} extended the result for $C(K)$ to the Banach space $C_0(L)$ of continuous functions on a locally compact Hausdorff space $L$ that vanish at infinity. 

In this paper, we study the above problem for the space $\CD$ of continuously differentiable functions on the closed unit interval $[0, 1]$ and the space $\Lip$ of Lipschitz continuous functions on $[0, 1]$. 
The space $\CD$ can be equipped with various norms. 
For example, the following norms are commonly considered on $\CD$: 
\begin{align*}
\Cp{f} & = 
\sup_{t \in [0, 1]} \sqrt[p]{|f(t)|^p + |f'(t)|^p}, \\
\Sigp{f} & = 
\begin{cases}
\sqrt[p]{\sn{f}^p + \sn{f'}^p} & (1 \le p < \infty), \\
\max\set{\sn{f}, \sn{f'}} & (p = \infty), 
\end{cases}
\\
\sigp{f} & = 
\begin{cases}
\sqrt[p]{|f(0)|^p + \sn{f'}^p} & (1 \le p < \infty), \\
\max\set{|f(0)|, \sn{f'}} & (p = \infty), 
\end{cases}
\end{align*}
where $\sn{f'}$ stands for the supremum norm of the derivative $f'$. 
The norms $\Sigm{\cdot}$ and $\sigm{\cdot}$ are also denoted by $\Mnorm{\cdot}$ and $\mnorm{\cdot}$, respectively. 
Equipped with each of the above norms, $\CD$ is a Banach space. 
By interpreting $\sn{f'}$ as the essential supremum norm of $f'$, the same formulas define norms on $\Lip$, making it a Banach space as well, except for $\Cp{\cdot}$. 

Isometries on these spaces have been investigated extensively; see, for example, \cite{CabCueEnaMiuPer, Cam, HirMiu, Kaw1, Kaw2, KawKosMiu, Kos1, Kos2, RaoRoy, Wan2}. 
Despite the extensive study of surjective isometries on $\CD$ and $\Lip$, little seems to be known about isometries between the positive parts of their unit spheres. 
The purpose of this paper is to investigate the positive-sphere version of Tingley's problem for these spaces equipped with the norm $\sigp{\cdot}$.


\subsection{Notations and Main Results}

For a Banach space $E$, the closed unit ball and the unit sphere of $E$ are denoted by $B(E)$ and $S(E)$, respectively. 
If $E$ is an ordered Banach space or the complexification of an ordered Banach space, the \textit{positive part} of the unit ball and unit sphere are defined as 
\begin{align*}
B^+(E) = \set{f \in B(E) \colon f \ge 0} \quad \text{and} \quad 
S^+(E) = \set{f \in S(E) \colon f \ge 0}, 
\end{align*}
respectively. 

Let $C(K)$ denote the Banach space of all continuous complex-valued functions on a compact Hausdorff space $K$, endowed with the supremum norm $\sn{u} = \sup_{t \in K} |u(t)|$. 
Let $\CR{K}$ be the real subspace of $C(K)$ consisting of all continuous real-valued functions on $K$. 
Note that $\CR{K}$ is an ordered Banach space with respect to the pointwise order, and $C(K)$ coincides with the complexification of $\CR{K}$. 
The symbols $\0$ and $\1$ denote the zero function $\0(t) = 0$ and the constant function $\1(t) = 1$, respectively. 

We shall consider the space $\CD$ of all continuously differentiable complex-valued functions on $[0, 1]$ and the space $\Lip$ of Lipschitz continuous complex-valued functions on $[0, 1]$, equipped with the norm
\begin{align*}
\sigp{f} & = 
\begin{cases}
\sqrt[p]{|f(0)|^p + \sn{f'}^p} & (1 \le p < \infty), \\
\max\set{|f(0)|, \sn{f'}} & (p = \infty). 
\end{cases}
\end{align*}
Note that every Lipschitz function $f \in \Lip$ is differentiable almost everywhere on the closed unit interval $[0, 1]$, and its derivative $f'$ belongs to $L^\infty[0, 1]$. 
Hence, for $f \in \Lip$, the symbol $\sn{f'}$ is interpreted as the essential supremum norm rather than the usual supremum norm. 
Let $\CDR$ and $\LipR$ be the real subspaces consisting of all real-valued $f \in \CD$ and all real-valued $f \in \Lip$, respectively. 
Both $\CDR$ and $\LipR$ are ordered Banach spaces with respect to the order defined by 
\begin{align*}
f \le g \quad 
\overset{\mathrm{def}}{\Longleftrightarrow} \quad 
f(0) \le g(0) \,\,\,\, \text{and} \,\,\,\, f' \le g' \,\, \text{on} \,\, [0, 1]. 
\end{align*}
Of course, for $f, g \in \LipR$, ``$f' \le g'$ on $[0, 1]$'' is understood in the almost everywhere sense. 
To avoid cumbersome notation, we introduce the following abbreviations for the positive parts of the unit spheres of these spaces: 
\begin{align*}
\PSsp & = S^+(\CD, \sigp{\cdot}) = S^+(\CDR, \sigp{\cdot}) \quad \text{and} \\
\PSLsp & = S^+(\Lip, \sigp{\cdot}) = S^+(\LipR, \sigp{\cdot}). 
\end{align*}

The following theorems are the main results of this paper: 

\begin{thm}\label{ThmMainC1}
Let $\Delta \colon \PSsp \to \PSsp$ be a surjective isometry. 
Then there exists a surjective complex-linear isometry $T \colon \CD \to \CD$ such that $T|_{\PSsp} = \Delta$. 
Moreover, $T$ is represented as  
\begin{align*}
T(f)(t) = f(0) + \int_0^t f'(\tau(x)) dx \qquad 
(f \in \CD, \,\, t \in [0, 1]), 
\end{align*}
where $\tau \colon [0, 1] \to [0, 1]$ is a homeomorphism. 
\end{thm}

\begin{thm}\label{ThmMainLip}
Let $\Delta \colon \PSLsp \to \PSLsp$ be a surjective isometry. 
Then there exists a surjective complex-linear isometry $T \colon \Lip \to \Lip$ such that $T|_{\PSLsp} = \Delta$. 
Moreover, $T$ is represented as  
\begin{align*}
T(f)(t) = f(0) + \int_0^t \Lambda(f')(x) dx \qquad 
(f \in \Lip, \,\, t \in [0, 1]), 
\end{align*}
where $\Lambda \colon L^\infty[0, 1] \to L^\infty[0, 1]$ is an isometric $\ast$-isomorphism. 
Furthermore, there exists a homeomorphism $\tau \colon \M \to \M$ such that 
\begin{align*}
\Lambda(u) = \Gamma^{- 1}(\Gamma(u) \circ \tau) \qquad 
(u \in L^\infty[0, 1]), 
\end{align*}
where $\M$ and $\Gamma \colon L^\infty[0, 1] \to C(\M)$ stand for the maximal ideal space and the Gelfand transform of $L^\infty[0, 1]$, respectively. 
\end{thm}

The study of Tingley's problem often relies on the fact, established by Cheng and Dong \cite{CheDon} and Tanaka \cite{Tan}, that every surjective isometry between unit spheres preserves maximal convex subsets. 
However, this result cannot be applied directly to surjective isometries between the positive parts of the unit spheres. 
When $p = \infty$, the spaces $\CD$ and $\Lip$ are isometrically isomorphic to the spaces of continuous functions on suitable compact Hausdorff spaces, and hence we can apply the result of Leung, Ng and Wong \cite{LeuNgWon} to prove our main results. 
The case $1 \le p < \infty$ is substantially more delicate and requires a detailed analysis of the geometric structure of the positive part of the unit sphere.
To this end, in Section~\ref{Sec2}, we investigate the structure of the positive part of the unit sphere by exploiting the fact that both $\CD$ and $\Lip$ admit decompositions as direct sums of $\C$ and the spaces of continuous functions on appropriate compact Hausdorff spaces. 
Based on these preparations, we prove our main results in Section~\ref{Sec3}. 
Finally, in Section~\ref{Sec4}, we present several consequences and remarks related to the main theorems.


\section{Surjective Isometries on $S^+({\mathbb{R} \oplus_p E})$}\label{Sec2}

In this section, we develop a unified framework for proving Theorems~\ref{ThmMainC1} and \ref{ThmMainLip} for $1 \le p < \infty$. 
Let $E$ be an ordered Banach space, and let $1 \le p \le \infty$. 
The $\ell^p$-sum $\R \oplus_p E$ is the direct sum linear space $\R \oplus E$ endowed with norm
\begin{align*}
\normp{(\a, u)} = 
\begin{cases}
\sqrt[p]{|\a|^p + \norm{u}^p} & (1 \le p < \infty), \\
\max\set{|\a|, \norm{u}} & (p = \infty). 
\end{cases}
\end{align*}
Observe that $\R \oplus_p E$ is also an ordered Banach space with the product order defined by 
\begin{align*}
(\a, u) \le (\b, v) \quad \quad 
\overset{\mathrm{def}}{\Longleftrightarrow} \quad 
\a \le \b \,\,\, \text{and} \,\,\, u \le v. 
\end{align*}
The following theorem is the main result of this section. 

\begin{thm}\label{ThmMain1}
Let $1 \le p < \infty$, let $E_1$ and $E_2$ be ordered Banach spaces with $\diam(\PBE{1}) = \diam(\PBE{2}) = 1$, where $\diam(A)$ stands for the diameter of $A$. 
Let $\Phi \colon \PSREp{1} \to \PSREp{2}$ be a surjective isometry with respect to the norm $\normp{\cdot}$. 
Then there exists a surjective isometry $\Phi' \colon \PBE{1} \to \PBE{2}$ such that $\Phi'(0) = 0$ and 
\begin{align}\label{EqThmMain1}
\Phi(\a, u) = (\a, \Phi'(u))  
\end{align}
for every $(\a, u) \in \PSREp{1}$. 
\end{thm}

The assumption $\diam(\PBE{}) = 1$ is satisfied, for example, when $E$ is the real-valued continuous function space $\CR{K}$ endowed with the pointwise order. 
As mentioned after the statement of our main results, the spaces $\CD$ and $\Lip$ can be decomposed as $\C \oplus_p C(K)$ for suitable compact Hausdorff spaces $K$. 
Thus the real parts $\CDR$ and $\LipR$ can be represented as $\R \oplus_p \CR{K}$. 
Therefore, Theorem~\ref{ThmMain1} provides a common framework for treating both spaces.

Throughout this section, let $E$ be an ordered Banach space with $\diam(\PBE{}) = 1$. 
To begin with, we introduce the slices of $\PSREp{}$ and characterize them in terms of the norm $\normp{\cdot}$. 

\begin{dfn}
For $0 \le r \le 1$, we define the \textit{$r$-slice} of $\PSREp{}$ to be the subset 
\begin{align*}
\PSREp{}(r) = \set{(\a, u) \in \PSREp{} \colon \a = r}. 
\end{align*}
\end{dfn}

\begin{lem}\label{LemSliceCharacterization}
Let $0 \le r \le 1$. 
The $r$-slice $\PSREp{}(r)$ is precisely the set of all $(\a, u) \in \PSREp{}$ whose distance from $(1, 0)$ is $\sqrt[p]{(1-r)^p + 1 - r^p}$: 
\begin{align*}
\PSREp{}(r) = \set{(\a, u) \in \PSREp{} : \normp{(\a, u) - (1, 0)} = \sqrt[p]{(1-r)^p + 1 - r^p}}. 
\end{align*}
\end{lem}

\begin{proof}
Let $(r, u) \in \PSREp{}(r)$. 
Then we have $r^p + \norm{u}^p = \normp{(r, u)}^p = 1$, and hence $\norm{u}^p = 1 - r^p$. 
Thus we obtain 
\begin{align*}
\normp{(r, u) - (1, 0)}^p 
= \normp{(r - 1, u)}^p 
= |r - 1|^p + \norm{u}^p 
= (1 - r)^p + 1 - r^p, 
\end{align*}
which proves that $\normp{(r, u) - (1, 0)} = \sqrt[p]{(1- r)^p + 1 - r^p}$. 

Conversely, let $(\a, u) \in \PSREp{}$, and assume that $\normp{(\a, u) - (1, 0)} = \sqrt[p]{(1- r)^p + 1 - r^p}$. 
Then we have 
\begin{align*}
(1 - \a)^p + \norm{u}^p 
= (1 - r)^p + 1 - r^p. 
\end{align*}
On the other hand, it follows from $\normp{(\a, u)} = 1$ that 
\begin{align*}
\a^p + \norm{u}^p 
= 1. 
\end{align*}
These two equations show that the value $\a$ must satisfy 
\begin{align*}
\a^p - (1 - \a)^p 
= r^p - (1 - r)^p. 
\end{align*}
Note that the function $x \mapsto x^p - (1 - x)^p$ is strictly increasing on the closed interval $[0, 1]$, and hence, it is one-to-one on $[0, 1]$. 
Thus we obtain $\a = r$, which proves that $(\a, u) \in \PSREp{}(r)$. 
\end{proof}

As an immediate consequence of Lemma~\ref{LemSliceCharacterization}, we see that 
\begin{align}\label{EqSlice1}
\PSREp{}(1) = \set{(1, 0)}. 
\end{align}

Our next goal is to show that every surjective isometry $\Phi \colon \PSREp{1} \to \PSREp{2}$ preserves the point $(1, 0)$ and the slices. 
To this end, we need the following lemma. 

\begin{lem}\label{LemAlternative}
Let $(\a, u), (\b, v) \in \PSREp{}$. 
If $\normp{(\a, u) - (\b, v)} = \sqrt[p]{2}$, then either $(\a, u) = (1, 0)$ or $(\b, v) = (1, 0)$ holds. 
\end{lem}

\begin{proof}
Assume that $\normp{(\a, u) - (\b, v)} = \sqrt[p]{2}$. 
Then we have 
\begin{align*}
|\a - \b|^p + \norm{u - v}^p = 2. 
\end{align*}
Since $0 \le \a, \b \le 1$ and $\diam(\PBE{}) = 1$, we see that $0 \le |\a - \b| \le 1$ and $0 \le \norm{u - v} \le 1$. 
Hence the above equality implies that $|\a - \b| = 1$. 
Having in mind that $0 \le \a, \b \le 1$, we see that either $\a = 1$ or $\b = 1$, that is, either $(\a, u) \in \PSREp{}(1)$ or $(\b, v) \in \PSREp{}(1)$. 
It thus follows from \eqref{EqSlice1} that either $(\a, u) = (1, 0)$ or $(\b, v) = (1, 0)$. 
\end{proof}

Until the end of this subsection, we fix ordered Banach spaces $E_1$ and $E_2$ with $\diam(\PBE{1}) = \diam(\PBE{2}) = 1$ and an arbitrary surjective isometry $\Phi \colon \PSREp{1} \to \PSREp{2}$ with respect to the norm $\normp{\cdot}$. 

\begin{lem}\label{LemPreservingUnity}
The surjective isometry $\Phi \colon \PSREp{1} \to \PSREp{2}$ preserves the point $(1, 0)$: $\Phi(1, 0) = (1, 0)$. 
\end{lem}

\begin{proof}
Choose distinct points $(\a, u), (\b, v) \in \PSREp{1}(0)$. 
Note that such points exist because $\diam(\PBE{1}) = 1$. 
It follows from Lemma~\ref{LemSliceCharacterization} that $\normp{(\a, u) - (1, 0)} = \normp{(\b, v) - (1, 0)} = \sqrt[p]{2}$. 
Since $\Phi$ is an isometry, we have 
\begin{align*}
\normp{\Phi(\a, u) - \Phi(1, 0)} = \normp{\Phi(\b, v) - \Phi(1, 0)} = \sqrt[p]{2}. 
\end{align*}
By applying Lemma~\ref{LemAlternative} to $\Phi(\a, u)$ and $\Phi(1, 0)$, we see that 
\begin{align*}
\text{either} \quad 
\Phi(\a, u) = (1, 0) 
\quad \text{or} \quad 
\Phi(1, 0) = (1, 0). 
\end{align*} 
Similarly, we have 
\begin{align*}
\text{either} \quad 
\Phi(\b, v) = (1, 0) 
\quad \text{or} \quad 
\Phi(1, 0) = (1, 0). 
\end{align*}
If we had $\Phi(1, 0) \ne (1, 0)$, then we would have $\Phi(\a, u) = (1, 0) = \Phi(\b, v)$. 
The injectivity of $\Phi$ would imply that $(\a, u) = (\b, v)$, which contradicts $(\a, u) \ne (\b, v)$. 
Therefore, we conclude that $\Phi(1, 0) = (1, 0)$. 
\end{proof}

\begin{lem}\label{LemPreservingSlices}
The surjective isometry $\Phi \colon \PSREp{1} \to \PSREp{2}$ preserves the slices: 
\begin{align*}
\Phi(\PSREp{1}(r)) = \PSREp{2}(r) 
\end{align*}
for every $0 \le r \le 1$. 
\end{lem}

\begin{proof}
Fix $(\a, u) \in \PSREp{1}$. 
Note that $\Phi(1, 0) = (1, 0)$ by Lemma~\ref{LemPreservingUnity}. 
Since $\Phi$ is an isometry, we have 
\begin{align*}
\normp{\Phi(\a, u) - (1, 0)} 
= \normp{\Phi(\a, u) - \Phi(1, 0)} 
= \normp{(\a, u) - (1, 0)}. 
\end{align*}
Hence Lemma~\ref{LemSliceCharacterization} shows that, for every $0 \le r \le 1$, $\Phi(\a, u) \in \PSREp{2}(r)$ if and only if $(\a, u) \in \PSREp{1}(r)$. 
Since $\Phi \colon \PSREp{1} \to \PSREp{2}$ is bijective, this proves $\Phi(\PSREp{1}(r)) = \PSREp{2}(r)$. 
\end{proof}

Note that $u \in \PBE{j}$ for every $(\a, u) \in \PSREp{j}$. 
For $j = 1, 2$, let $\pi_j \colon \PSREp{j} \to \PBE{j}$ be the map defined by 
\begin{align*}
\pi_j(\a, u) = u 
\end{align*}
for every $(\a, u) \in \PSREp{j}$. 

\begin{prop}\label{PropProjection}
The map $\pi_j \colon \PSREp{j} \to \PBE{j}$ is bijective, and its inverse is given by 
\begin{align*}
\pi_j^{- 1}(v) = (\sqrt[p]{1 - \norm{v}^p}, v) 
\end{align*}
for every $v \in \PBE{j}$. 
\end{prop}

\begin{proof}
For each $v \in \PBE{j}$, we set $\Pi_j(v) = (\sqrt[p]{1 - \norm{v}^p}, v)$. 
Note that $\Pi_j(v) \in \PSREp{j}$, since $\sqrt[p]{1 - \norm{v}^p} \ge 0$, $v \ge 0$ and 
\begin{align*}
\normp{\Pi_j(v)} 
= \normp{(\sqrt[p]{1 - \norm{v}^p}, v)} 
= \sqrt[p]{(1 - \norm{v}^p) + \norm{v}^p} 
= 1. 
\end{align*}
Hence we obtain a well-defined map $\Pi_j \colon \PBE{j} \to \PSREp{j}$. 
It is clear that $(\pi_j \circ \Pi_j)(v) = v$ for every $v \in \PBE{j}$, and thus $\pi_j \circ \Pi_j = \id_{\PBE{j}}$. 

Fix $(\a, u) \in \PSREp{j}$. 
Since $\a^p + \norm{u}^p = \normp{(\a, u)}^p = 1$, we see that $\sqrt[p]{1 - \norm{u}^p} = \a$. 
Hence we have 
\begin{align*}
(\Pi_j \circ \pi_j)(\a, u) 
= \Pi_j(u) 
= (\sqrt[p]{1 - \norm{u}^p}, u) 
= (\a, u) 
\end{align*}
which shows that $\Pi_j \circ \pi_j = \id_{\PSREp{j}}$. 
Thus, we conclude that $\pi_j \colon \PSREp{j} \to \PBE{j}$ is bijective and $\pi_j^{-1} = \Pi_j$. 
\end{proof}

We are now ready to prove Theorem~\ref{ThmMain1}. 

\begin{proof}[\textbf{Proof of Theorem~\ref{ThmMain1}}]
Define a map $\Phi' \colon \PBE{1} \to \PBE{2}$ by 
\begin{align}\label{EqInducedIsometry}
\Phi' = \pi_2 \circ \Phi \circ \pi_1^{-1}. 
\end{align}
The surjectivity of $\Phi'$ follows immediately from the definition. 
Since $\pi_1^{-1}(0) = (1, 0)$, it follows from Lemma~\ref{LemPreservingUnity} that 
\begin{align*}
\Phi'(0) 
= (\pi_2 \circ \Phi \circ \pi_1^{-1})(0) 
= \pi_2(\Phi(1, 0)) 
= \pi_2(1, 0) 
= 0. 
\end{align*}

Let $(\a, u) \in \PSREp{1}$. 
Since $\Phi = \pi_2^{-1} \circ \Phi' \circ \pi_1$, we see that 
\begin{align*}
\Phi(\a, u) 
= (\pi_2^{-1} \circ \Phi')(u) 
= \pi_2^{-1}(\Phi'(u)) 
= (\sqrt[p]{1 - \norm{\Phi'(u)}^p}, \Phi'(u)). 
\end{align*}
On the other hand, Lemma~\ref{LemPreservingSlices} implies that $\Phi(\a, u) \in \Phi(\PSREp{1}(\a)) = \PSREp{2}(\a)$, and hence we must have $\sqrt[p]{1 - \norm{\Phi'(u)}^p} = \a$. 
Thus we conclude that \eqref{EqThmMain1} holds: 
\begin{align}\label{PfEqThmMain1}
\Phi(\a, u) = (\a, \Phi'(u)) 
\end{align}
for every $(\a, u) \in \PSREp{1}$. 

It remains to show that $\Phi' \colon \PBE{1} \to \PBE{2}$ is an isometry. 
To see this, fix $u_0, u_1 \in \PBE{1}$. 
Since $\Phi \colon \PSREp{1} \to \PSREp{2}$ is a surjective isometry, we have 
\begin{align}\label{PfEqThmMain2}
\normp{\Phi(\pi_1^{-1}(u_0)) - \Phi(\pi_1^{-1}(u_1))}^p = \normp{\pi_1^{-1}(u_0) - \pi_1^{-1}(u_1)}^p. 
\end{align}
By Proposition~\ref{PropProjection}, we have $\pi_1^{-1}(u_k) = (\sqrt[p]{1 - \norm{u_k}^p}, u_k)$ for $k = 0, 1$. 
For simplicity of notation, we set $\a_k = \sqrt[p]{1 - \norm{u_k}^p}$. 
Then $\pi_1^{-1}(u_k) = (\a_k, u_k)$. 
Moreover, by \eqref{PfEqThmMain1}, we have $\Phi(\pi_1^{-1}(u_k)) = \Phi(\a_k, u_k) = (\a_k, \Phi'(u_k))$. 
Substituting these equalities into \eqref{PfEqThmMain2}, we obtain $\normp{(\a_0, \Phi'(u_0)) - (\a_1, \Phi'(u_1))}^p = \normp{(\a_0, u_0) - (\a_1, u_1)}^p$, which yields 
\begin{align*}
|\a_0 - \a_1|^p + \norm{\Phi'(u_0) - \Phi'(u_1)}^p 
= |\a_0 - \a_1|^p + \norm{u_0 - u_1}^p. 
\end{align*}
This gives that $\norm{\Phi'(u_0) - \Phi'(u_1)} = \norm{u_0 - u_1}$, which shows that $\Phi' \colon \PBE{1} \to \PBE{2}$ is an isometry. 
The proof is complete. 
\end{proof}


\section{Proof of Main Results}\label{Sec3}

In this section, we prove our two main results, Theorems~\ref{ThmMainC1} and \ref{ThmMainLip}. 
To this end, we shall prove the following key lemma: 

\begin{lem}\label{LemKey}
Let $1 \le p \le \infty$, and let $K$ and $L$ be compact Hausdorff spaces without isolated points. 
If $\Phi \colon \PSRCp{K} \to \PSRCp{L}$ is a surjective isometry, then there exists a homeomorphism $\tau \colon L \to K$ such that 
\begin{align*}
\Phi(\a, u) = (\a, u \circ \tau) 
\end{align*}
for every $(\a, u) \in \PSRCp{K}$. 
\end{lem}

To prove Lemma~\ref{LemKey} for $1 \le p < \infty$, we need to make an observation on the positive part $\PBCR{K}$ of the unit ball of $\CR{K}$. 
In this section, the symbols $K$ and $L$ stand for compact Hausdorff spaces. 

A subset of a Banach space is called a \textit{convex body} if it is a closed convex set with non-empty interior. 
We recall the following theorem due to Mankiewicz. 

\begin{prop}[{\cite[Theorem~5]{Man}}]\label{PropMan}
Let $V$ and $W$ be convex bodies in real Banach spaces $E$ and $F$, respectively. 
Then every surjective isometry $\Psi \colon V \to W$ is uniquely extended to a surjective affine isometry $\widetilde{\Psi} \colon E \to F$. 
\end{prop}

To apply Mankiewicz's theorem, we need to show that $\PBCR{K}$ is a convex body in $\CR{K}$. 

\begin{lem}\label{LemConvexBody}
The positive part $\PBCR{K}$ of the unit ball of $\CR{K}$ is a convex body in $\CR{K}$. 
\end{lem}

\begin{proof}
It is straightforward to verify that $\PBCR{K}$ is closed and convex. 
If $u \in \CR{K}$ satisfies $\sn{u - \frac{1}{2} \1} < \frac{1}{4}$, then we have $\frac{1}{4} < u(x) < \frac{3}{4}$ for every $x \in K$, which shows that $u \in \PBCR{K}$. 
Hence the function $\frac{1}{2} \1$ is an interior point of $\PBCR{K}$. 
Consequently, $\PBCR{K}$ is a convex body in $\CR{K}$. 
\end{proof}

To prove Lemma~\ref{LemKey} for $p = \infty$, we need the following result due to Leung, Ng and Wong. 

\begin{prop}[{\cite[Theorem~15]{LeuNgWon}}]\label{PropLNW}
Let $X$ and $Y$ be compact Hausdorff spaces, and let $\Psi \colon \PSCR{X} \to \PSCR{Y}$ be a surjective isometry. 
Then there exists a homeomorphism $\rho \colon Y \to X$ such that 
\begin{align*}
\Psi(u) = u \circ \rho
\end{align*}
for every $u \in \PSCR{X}$. 
\end{prop}

In order to apply the Leung--Ng--Wong Theorem (Proposition~\ref{PropLNW}) to show Lemma~\ref{LemKey} for $p = \infty$, we need to construct an isometric isomorphism from $\R \oplus_\infty \CR{K}$ onto $\CR{\tilde{K}}$ for some compact Hausdorff space $\tilde{K}$. 

Consider the direct sum topological space $\{x_0\} \cup K$, where $x_0$ is a point distinguished from $K$. 
Endowed with the direct sum topology, $\{x_0\} \cup K$ is a compact Hausdorff space. 
Since $K$ has no isolated points, the distinguished point $x_0$ is a unique isolated point of $\{x_0\} \cup K$. 

For each $(\a, u) \in \R \oplus_\infty \CR{K}$, we define a function $J_K(\a, u) \colon \{x_0\} \cup K \to \R$ by 
\begin{align}\label{EqIOI}
J_K(\a, u)(x) = 
\begin{cases}
\a & (x = x_0), \\
u(x) & (x \in K). 
\end{cases}
\end{align}
It is clear that $J_K(\a, u) \in \CR{\{x_0\} \cup K}$ for all $(\a, u) \in \R \oplus_\infty \CR{K}$. 
Thus this defines a map $J_K \colon \R \oplus_\infty \CR{K} \to \CR{\{x_0\} \cup K}$. 
The proof of the following proposition is routine, so we omit its proof. 

\begin{prop}\label{PropIsomOrdIsom}
The map $J_K \colon \R \oplus_\infty \CR{K} \to \CR{\{x_0\} \cup K}$ is an isometric order isomorphism, whose inverse is given by 
\begin{align}\label{EqInvIOI}
J_K^{- 1}(w) = (w(x_0), w|_K) 
\end{align}
for all $w \in \CR{\{x_0\} \cup K}$. 
In particular, $J_K(\PSRCm{K}) = \PSCR{\{x_0\} \cup K}$. 
\end{prop}

To avoid notational complexity, we write the restriction $J_K|_{\PSRCm{K}} \colon \PSRCm{K} \to \PSCR{\{x_0\} \cup K}$ simply as $J_K$. 
By Proposition~\ref{PropIsomOrdIsom}, the restriction $J_K \colon \PSRCm{K} \to \PSCR{\{x_0\} \cup K}$ is a surjective isometry, whose inverse is given by \eqref{EqInvIOI} for every $w \in \PSCR{\{x_0\} \cup K}$. 
This construction allows us to reduce the case $p = \infty$
to an application of the Leung--Ng--Wong theorem.

We are now ready to prove Lemma~\ref{LemKey}. 
The proof is divided into two parts: the case $1 \le p < \infty$ and the case $p = \infty$. 

\begin{proof}[\textbf{Proof of Lemma~\ref{LemKey}}]
We first prove Lemma~\ref{LemKey} for $1 \le p < \infty$. 
Let $\Phi \colon \PSRCp{K} \to \PSRCp{L}$ be a surjective isometry. 
Recalling that $\diam(\PBCR{K}) = 1$, we can apply Theorem~\ref{ThmMain1} to obtain a surjective isometry $\Phi' \colon \PBCR{K} \to \PBCR{L}$ with $\Phi'(\0) = \0$ such that
\begin{align}\label{EqKey0}
\Phi(\a, u) = (\a, \Phi'(u)) 
\end{align}
for every $(\a, u) \in \PSRCp{K}$. 

We show that there exists a homeomorphism $\tau \colon L \to K$ such that 
\begin{align}\label{EqKey1}
\Phi'(u) = u \circ \tau 
\end{align}
for every $u \in \PBCR{K}$. 
Note that, by Lemma~\ref{LemConvexBody}, $\PBCR{K}$ and $\PBCR{L}$ are convex bodies in $\CR{K}$ and $\CR{L}$, respectively. 
By Mankiewicz's theorem (Proposition~\ref{PropMan}), $\Phi' \colon \PBCR{K} \to \PBCR{L}$ extends to a surjective affine isometry $\widetilde{\Phi'} \colon \CR{K} \to \CR{L}$. 
Since $\widetilde{\Phi'}(\0) = \Phi'(\0) = \0$, the isometry $\widetilde{\Phi'}$ must be linear. 
The Banach--Stone theorem guarantees the existence of a continuous function $\gamma \colon L \to \set{1, -1}$ and a homeomorphism $\tau \colon L \to K$ such that 
\begin{align*}
\widetilde{\Phi'}(u)(y) = \gamma(y) u(\tau(y)) 
\end{align*}
for all $u \in \CR{K}$ and $y \in L$. 
Note that $\widetilde{\Phi'}(\1) = \Phi'(\1) \in \PBCR{L}$ since $\1 \in \PBCR{K}$. 
Evaluating the above equality at $u = \1$, we see that $\gamma(y) = \Phi'(\1)(y) \ge 0$, which shows that $\gamma(y) = 1$ for all $y \in L$. 
Thus we have $\widetilde{\Phi'}(u) = u \circ \tau$ for every $u \in \CR{K}$. 
By restricting this equality to $\PBCR{K}$, we obtain \eqref{EqKey1}. 
Combining \eqref{EqKey0} with \eqref{EqKey1}, we therefore conclude that Lemma~\ref{LemKey} holds for $1 \le p < \infty$. 

We next prove Lemma~\ref{LemKey} for $p = \infty$. 
Let $\Phi \colon \PSRCm{K} \to \PSRCm{L}$ be a surjective isometry. 
Consider the direct sum topological spaces $\{x_0\} \cup K$ and $\{y_0\} \cup L$ and the surjective isometry $J_K \colon \PSRCm{K} \to \PSCR{\{x_0\} \cup K}$ and $J_L \colon \PSRCm{L} \to \PSCR{\{y_0\} \cup L}$ defined by \eqref{EqIOI}. 
The composition 
\begin{align*}
\cc{\Phi} = J_L \circ \Phi \circ J_K^{-1} \colon \PSCR{\{x_0\} \cup K} \to \PSCR{\{y_0\} \cup L} 
\end{align*}
is also a surjective isometry. 
Applying the Leung--Ng--Wong Theorem (Proposition~\ref{PropLNW}) to $\cc{\Phi}$, there exists a homeomorphism $\rho \colon \{y_0\} \cup L \to \{x_0\} \cup K$ such that 
\begin{align}\label{EqKey2}
\cc{\Phi}(w) = w \circ \rho 
\end{align}
for every $w \in \CR{\{x_0\} \cup K}$. 
Since $K$ and $L$ have no isolated points, the points $x_0$ and $y_0$ are unique isolated points of $\{x_0\} \cup K$ and $\{y_0\} \cup L$, respectively. 
Thus we see that $\rho(y_0) = x_0$ and $\rho(L) = K$. 
It follows that the restriction $\tau = \rho|_L \colon L \to K$ is a homeomorphism. 

By the definition of $\cc{\Phi}$, we have $\Phi = J_L^{- 1} \circ \cc{\Phi} \circ J_K$. 
Fix $(\a, u) \in \PSRCm{K}$ arbitrarily. 
It follows from \eqref{EqInvIOI} and \eqref{EqKey2} that 
\begin{align*}
\Phi(\a, u) 
= (J_L^{- 1} \circ \cc{\Phi})(J_K(\a, u)) 
= J_L^{- 1}(J_K(\a, u) \circ \rho). 
\end{align*}
Since $\rho(y_0) = x_0$, we have $(J_K(\a, u) \circ \rho)(y_0) = J_K(\a, u)(x_0) = \a$ by \eqref{EqIOI}. 
If $y \in L$, then $\rho(y) = \tau(y) \in K$, and hence $(J_K(\a, u) \circ \rho)(y) = J_K(\a, u)(\tau(y)) = u(\tau(y)) = (u \circ \tau)(y)$. 
It therefore follows from \eqref{EqInvIOI} that 
\begin{align*}
\Phi(\a, u) 
= J_L^{- 1}(J_K(\a, u) \circ \rho) 
= (\a, u \circ \tau), 
\end{align*}
as desired. 
The proof is complete. 
\end{proof}

We now proceed to the proof of Theorems~\ref{ThmMainC1} and \ref{ThmMainLip}. 
First, we consider the space $\CD$ of continuously differentiable functions. 

We define a map $D \colon \CD \to \C \oplus \CC$ by 
\begin{align*}
D(f) = (f(0), f') 
\end{align*}
for every $f \in \CD$. 

\begin{prop}\label{PropC1D}
The map $D \colon \CD \to \C \oplus \CC$ has the following properties: 
\begin{enumerate}[(i)]
\item \label{EqPropC1D1} 
$D$ is a bijective complex-linear operator, whose inverse is given by 
\begin{align*}
D^{- 1}(\a, u)(t) 
= \a + \int_0^t u(x) dx 
\end{align*}
for every $(\a, u) \in \C \oplus \CC$ and $t \in [0, 1]$; 

\item \label{EqPropC1D2} 
$D(\CDR) = \R \oplus \CCR$ and $D|_{\CDR} \colon \CDR \to \R \oplus \CCR$ is an order isomorphism; 

\item \label{EqPropC1D3} 
For every $1 \le p \le \infty$, $D$ is an isometry from $(\CD, \sigp{\cdot})$ onto $\C \oplus_p \CC$. 
\end{enumerate}
\end{prop}

\begin{proof}
It is clear that $D$ is complex-linear. 
We show that $D$ is an isometry from $(\CD, \sigp{\cdot})$ onto $\C \oplus_p \CC$ for every $1 \le p \le \infty$. 
Let $f \in \CD$. 
If $1 \le p < \infty$, then we have 
\begin{align*}
\normp{D(f)}^p 
= \normp{(f(0), f')}^p 
= |f(0)|^p + \sn{f'}^p 
= \sigp{f}^p, 
\end{align*}
and hence $\normp{D(f)} = \sigp{f}$. 
For $p = \infty$, we see that 
\begin{align*}
\normm{D(f)} 
= \max\set{|f(0)|, \sn{f'}} 
= \sigm{f} 
\end{align*}
In either case, $D$ is an isometry, which establishes \eqref{EqPropC1D3}. 
In particular, $D$ must be injective. 

Define $I \colon \C \oplus \CC \to \CD$ by 
\begin{align*}
I(\a, u)(t) = \a + \int_0^t u(x) dx 
\end{align*}
for every $(\a, u) \in \C \oplus \CC$ and $t \in [0, 1]$. 
To prove \eqref{EqPropC1D1}, it remains to show that $D$ is surjective and $D^{-1} = I$. 
If $(\a, u) \in \C \oplus \CC$, then we have $I(\a, u)(0) = \a$ and $I(\a, u)' = u$ by the fundamental theorem of calculus. 
Thus we obtain 
\begin{align*}
(D \circ I)(\a, u) 
= (I(\a, u)(0), I(\a, u)') 
= (\a, u), 
\end{align*}
which proves that $D \circ I = \id_{\C \oplus \CC}$. 
Consequently, we conclude \eqref{EqPropC1D1}. 

By the definitions of $D$ and \eqref{EqPropC1D1}, it is clear that $D(\CDR) = \R \oplus \CCR$. 
By the definitions of orders on $\CDR$ and $\R \oplus \CCR$, for every $f, g \in \CDR$, we have 
\begin{align*}
f \le g \quad \eq \quad 
(f(0), f') \le (g(0), g') \quad \eq \quad 
D(f) \le D(g). 
\end{align*}
Since $D$ is bijective and $D(\CDR) = \R \oplus \CCR$, it is an order isomorphism. 
Hence \eqref{EqPropC1D2} is established. 
\end{proof}

It follows from Proposition~\ref{PropC1D} that $D$ maps $\PSsp = S^+(\CD, \sigp{\cdot}) = S^+(\CDR, \sigp{\cdot})$ onto $\PSCCp{[0, 1]} = \PSRCp{[0, 1]}$: 
\begin{align*}
D(\PSsp) = \PSRCp{[0, 1]}. 
\end{align*}
Moreover, the restriction $D|_{\PSsp} \colon \PSsp \to \PSRCp{[0, 1]}$ is a surjective isometry. 
For simplicity, we write it as $D_p \colon \PSsp \to \PSRCp{[0, 1]}$. 

\begin{proof}[\textbf{Proof of Theorem~\ref{ThmMainC1}}]
Fix $1 \le p \le \infty$. 
Since $D_p \colon \PSsp \to \PSRCp{[0, 1]}$ is a surjective isometry, the composition $\Phi_0 = D_p \circ \Delta \circ D_p^{-1} \colon \PSRCp{[0, 1]} \to \PSRCp{[0, 1]}$ is also a surjective isometry on $\PSRCp{[0, 1]}$. 
Note that $[0, 1]$ has no isolated points. 
Hence Lemma~\ref{LemKey} guarantees the existence of a homeomorphism $\tau \colon [0, 1] \to [0, 1]$ such that 
\begin{align*}
\Phi_0(\a, u) = (\a, u \circ \tau) 
\end{align*}
for every $(\a, u) \in \PSRCp{[0, 1]}$. 

Fix $f \in \PSsp$. 
Since $\Delta = D_p^{-1} \circ \Phi_0 \circ D_p$, we have 
\begin{align*}
\Delta(f) = (D_p^{-1} \circ \Phi_0)(f(0), f') = D_p^{-1}(f(0), f' \circ \tau). 
\end{align*}
It follows from \eqref{EqPropC1D1} of Proposition~\ref{PropC1D} that 
\begin{align*}
\Delta(f)(t) 
= D_p^{-1}(f(0), f' \circ \tau)(t) 
= f(0) + \int_0^t f'(\tau(x)) dx 
\end{align*}
for every $t \in [0, 1]$. 

We now define $T \colon \CD \to \CD$ by 
\begin{align*}
T(g)(t) = g(0) + \int_0^t g'(\tau(x)) dx 
\end{align*}
for every $g \in \CD$ and $t \in [0, 1]$. 
A straightforward computation shows that $T$ is a complex-linear isometry such that $T|_{\PSsp} = \Delta$. 
For any $h_1 \in \CD$, we define $g_1 \colon [0, 1] \to \C$ by 
\begin{align*}
g_1(t) = h_1(0) + \int_0^t h_1'(\tau^{- 1}(x)) dx 
\end{align*}
for every $t \in [0, 1]$. 
Then we see that $g_1 \in \CD$. 
Moreover, we have $g_1(0) = h_1(0)$ and $g_1' = h_1' \circ \tau^{- 1}$, and hence 
\begin{align*}
T(g_1)(t) 
& = g_1(0) + \int_0^t g_1'(\tau(x)) dx 
= h_1(0) + \int_0^t h_1'(x) dx 
= h_1 (t). 
\end{align*}
This shows the surjectivity of $T$. 
Finally, if $g \in \CD$, then we have $T(g)(0) = g(0)$ and $T(g)' = g' \circ \tau$ by the definition of $T$. 
If $1 \le p < \infty$, then we obtain 
\begin{align*}
\sigp{T(g)} 
= \sqrt[p]{|g(0)|^p + \sn{g' \circ \tau}^p} 
= \sqrt[p]{|g(0)|^p + \sn{g'}^p} 
= \sigp{g}. 
\end{align*}
For $p = \infty$, we similarly obtain 
\begin{align*}
\sigm{T(g)} 
= \max\set{|g(0)|, \sn{g' \circ \tau}} 
= \max\set{|g(0)|, \sn{g'}}
= \sigm{g}. 
\end{align*}
Therefore, $T$ is an isometry with respect to the norm $\sigp{\cdot}$. 
The proof is complete. 
\end{proof}

The proof of Theorem~\ref{ThmMainLip} is, for the most part, the same as that of Theorem~\ref{ThmMainC1}. 
Hence we shall concentrate on the differences from the previous proof. 

Recall that every $f \in \Lip$ is differentiable almost everywhere on the closed unit interval $[0, 1]$, and its derivative $f'$ belongs to $L^\infty[0, 1]$. 
Note that $L^\infty[0, 1]$ is a unital commutative $C^\ast$-algebra. 
Let $\M$ be the maximal ideal space of $L^\infty[0, 1]$, and let $\Gamma \colon L^\infty[0, 1] \to C(\M)$ be the Gelfand transform. 
Note that $\Gamma$ is an isometric $\ast$-isomorphism. 
Note also that $\Gamma(L^\infty_\R[0, 1]) = \CR{\M}$ and $\Gamma|_{L^\infty_\R[0, 1]} \colon L^\infty_\R[0, 1] \to \CR{\M}$ is an isometric order isomorphism. 

We define a map $\DL \colon \Lip \to \C \oplus C(\M)$ by 
\begin{align*}
\DL(f) = (f(0), \Gamma(f')) 
\end{align*}
for every $f \in \Lip$. 

\begin{prop}\label{PropLipD}
The map $\DL \colon \Lip \to \C \oplus C(\M)$ has the following properties: 
\begin{enumerate}[(i)]
\item \label{EqPropLipD1} 
$\DL$ is a bijective complex-linear operator, whose inverse is given by 
\begin{align*}
\DL^{- 1}(\a, u)(t) 
= \a + \int_0^t \Gamma^{- 1}(u)(x) dx 
\end{align*}
for every $(\a, u) \in \C \oplus C(\M)$ and $t \in [0, 1]$; 

\item \label{EqPropLipD2} 
$\DL(\LipR) = \R \oplus \CR{\M}$ and $\DL|_{\LipR} \colon \LipR \to \R \oplus \CR{\M}$ is an order isomorphism; 

\item \label{EqPropLipD3} 
For every $1 \le p \le \infty$, $\DL$ is an isometry from $(\Lip, \sigp{\cdot})$ onto $\C \oplus_p C(\M)$. 
\end{enumerate}
\end{prop}

\begin{proof}
It is clear that $\DL$ is complex-linear. 
Having in mind that the Gelfand transform $\Gamma \colon L^\infty[0, 1] \to C(\M)$ is a linear isometry, it can be shown by calculations similar to those in the proof of Proposition~\ref{PropC1D} that $\normp{\DL(f)} = \sigp{f}$ for every $f \in \Lip$ and $1 \le p \le \infty$. 
Hence we obtain \eqref{EqPropLipD3}. 
In particular, $\DL$ must be injective. 

Define $\IL \colon \C \oplus C(\M) \to \Lip$ by 
\begin{align*}
\IL(\a, u)(t) = \a + \int_0^t \Gamma^{- 1}(u)(x) dx 
\end{align*}
for every $(\a, u) \in \C \oplus C(\M)$ and $t \in [0, 1]$. 
A calculation similar to Proposition~\ref{PropC1D} shows that $\DL \circ \IL = \id_{\C \oplus C(\M)}$. 
Indeed, if $(\a, u) \in \C \oplus C(\M)$, then we have $\IL(\a, u)(0) = \a$ and $\IL(\a, u)' = \Gamma^{-1}(u)$ almost everywhere on $[0, 1]$, and thus 
\begin{align*}
(\DL \circ \IL)(\a, u) 
= (\IL(\a, u)(0), \Gamma(\IL(\a, u)')) 
= (\a, \Gamma(\Gamma^{-1}(u))) 
= (\a, u). 
\end{align*}
Hence $\DL$ is surjective and $\DL^{- 1} = \IL$, which shows \eqref{EqPropLipD1}. 

Recall that $\Gamma(L^\infty_\R[0, 1]) = \CR{\M}$ and $\Gamma|_{L^\infty_\R[0, 1]} \colon L^\infty_\R[0, 1] \to \CR{\M}$ is an isometric order isomorphism. 
Then we obtain $\DL(\LipR) = \R \oplus \CR{\M}$. 
By the definitions of orders on $\LipR$ and $\R \oplus \CR{\M}$, for every $f, g \in \LipR$, we have 
\begin{align*}
f \le g \quad 
& \eq \quad (f(0), f') \le (g(0), g') \\
& \eq \quad (f(0), \Gamma(f')) \le (g(0), \Gamma(g')) \\
& \eq \quad \DL(f) \le \DL(g). 
\end{align*}
Since $\DL$ is bijective and $\DL(\LipR) = \R \oplus \CR{\M}$, it is an order isomorphism. 
Hence \eqref{EqPropLipD2} is established. 
\end{proof}

It follows from Proposition~\ref{PropLipD} that $\DL$ maps $\PSLsp = S^+(\Lip, \sigp{\cdot}) = S^+(\LipR, \sigp{\cdot})$ onto $\PSCCp{\M} = \PSRCp{\M}$: 
\begin{align*}
\DL(\PSLsp) = \PSRCp{\M}. 
\end{align*}
Moreover, the restriction $\DL|_{\PSLsp} \colon \PSLsp \to \PSRCp{\M}$ is a surjective isometry. 
For simplicity, we write it as $\DL_p \colon \PSLsp \to \PSRCp{\M}$. 

In contrast to the proof of Theorem~\ref{ThmMainC1}, it appears not to be trivial that the maximal ideal space $\M$ of $L^\infty[0, 1]$ has no isolated points. 
Fortunately, it is true. 
It may be well-known, but we give a proof for the sake of completeness. 

\begin{prop}\label{PropIsolatedM}
The maximal ideal space $\M$ has no isolated points. 
\end{prop}

\begin{proof}
Suppose that $\M$ has an isolated point $q$. 
Then the characteristic function $\chi_{\set{q}}$ is a minimal non-zero projection in $C(\M)$. 
Since $\Gamma \colon L^\infty[0, 1] \to C(\M)$ is a $\ast$-isomorphism, $\Gamma^{-1}(\chi_{\set{q}})$ is a minimal non-zero projection in $L^\infty[0, 1]$. 

On the other hand, every projection in $L^\infty[0, 1]$ is a characteristic function of a measurable subset of $[0, 1]$ with positive Lebesgue measure. 
Since such a set can be decomposed into two disjoint measurable subsets of positive measure, the projection $\Gamma^{-1}(\chi_{\set{q}})$ in $L^\infty[0, 1]$ cannot be minimal.
Therefore, $\M$ has no isolated points. 
\end{proof}

The proof of Theorem~\ref{ThmMainLip} is parallel to that of
Theorem~\ref{ThmMainC1}, and we only indicate the necessary modifications. 

\begin{proof}[\textbf{Proof of Theorem~\ref{ThmMainLip}}]
Fix $1 \le p \le \infty$. 
Since $\DL_p \colon \PSLsp \to \PSRCp{\M}$ is a surjective isometry, the composition $\Phi_0 = \DL_p \circ \Delta \circ \DL_p^{-1} \colon \PSRCp{\M} \to \PSRCp{\M}$ is also a surjective isometry on $\PSRCp{\M}$. 
By Proposition~\ref{PropIsolatedM}, $\M$ has no isolated points. 
Hence Lemma~\ref{LemKey} guarantees the existence of a homeomorphism $\tau \colon \M \to \M$ such that 
\begin{align*}
\Phi_0(\a, u) = (\a, u \circ \tau) 
\end{align*}
for every $(\a, u) \in \PSRCp{\M}$. 

Fix $f \in \PSLsp$. 
Since $\Delta = \DL_p^{-1} \circ \Phi_0 \circ \DL_p$, we have 
\begin{align*}
\Delta(f) = (\DL_p^{-1} \circ \Phi_0)(f(0), \Gamma(f')) = \DL_p^{-1}(f(0), \Gamma(f') \circ \tau). 
\end{align*}
It follows from \eqref{EqPropLipD1} of Proposition~\ref{PropLipD} that 
\begin{align*}
\Delta(f)(t) 
= \DL_p^{-1}(f(0), \Gamma(f') \circ \tau)(t) 
= f(0) + \int_0^t \Gamma^{- 1}(\Gamma(f') \circ \tau)(x) dx 
\end{align*}
for every $t \in [0, 1]$. 

Define $\Lambda \colon L^\infty[0, 1] \to L^\infty[0, 1]$ by 
\begin{align*}
\Lambda(u) = \Gamma^{- 1}(\Gamma(u) \circ \tau) 
\end{align*}
for every $u \in L^\infty[0, 1]$. 
Both the Gelfand transform $\Gamma \colon L^\infty[0, 1] \to C(\M)$ and the composition operator $C(\M) \ni v \mapsto v \circ \tau \in C(\M)$ are isometric $\ast$-isomorphisms between $C^\ast$-algebras, and hence so is $\Lambda$. 

We now define $T \colon \Lip \to \Lip$ by 
\begin{align*}
T(g)(t) = g(0) + \int_0^t \Lambda(g')(x) dx 
\end{align*}
for every $g \in \Lip$ and $t \in [0, 1]$. 
A straightforward computation shows that $T$ is a complex-linear isometry such that $T|_{\PSLsp} = \Delta$. 
For any $h_1 \in \Lip$, we define $g_1 \colon [0, 1] \to \C$ by 
\begin{align*}
g_1(t) = h_1(0) + \int_0^t \Lambda^{- 1}(h_1')(x) dx 
\end{align*}
for every $t \in [0, 1]$. 
Then we see that $g_1 \in \Lip$. 
Moreover, we have $g_1(0) = h_1(0)$ and $g_1' = \Lambda^{- 1}(h_1')$ almost everywhere on $[0, 1]$, and hence 
\begin{align*}
T(g_1)(t) 
& = g_1(0) + \int_0^t \Lambda(g_1')(x) dx 
= h_1(0) + \int_0^t h_1'(x) dx 
= h_1 (t). 
\end{align*}
This shows the surjectivity of $T$. 

Finally, if $g \in \Lip$, then we have $T(g)(0) = g(0)$ and $T(g)' = \Lambda(g')$ almost everywhere on $[0, 1]$ by the definition of $T$. 
If $1 \le p < \infty$, then we obtain 
\begin{align*}
\sigp{T(g)} 
= \sqrt[p]{|g(0)|^p + \sn{\Lambda(g')}^p} 
= \sqrt[p]{|g(0)|^p + \sn{g'}^p} 
= \sigp{g}. 
\end{align*}
For $p = \infty$, we similarly obtain 
\begin{align*}
\sigm{T(g)} 
= \max\set{|g(0)|, \sn{\Lambda(g')}} 
= \max\set{|g(0)|, \sn{g'}}
= \sigm{g}. 
\end{align*}
Therefore, $T$ is an isometry with respect to the norm $\sigp{\cdot}$. 
The proof is complete. 
\end{proof}


\section{Corollaries and Remarks}\label{Sec4}

The descriptions obtained in Theorems~\ref{ThmMainC1} and \ref{ThmMainLip} immediately yield extension results for surjective isometries on the positive part of the unit sphere.
In particular, such isometries are induced by isometric order isomorphisms of the corresponding real function spaces.


\subsection{Extension to Isometric Order Isomorphisms}

The following corollaries show that surjective isometries on the positive part of the unit sphere extend to isometric order isomorphisms of $\CDR$ and $\LipR$. 
We first consider the case of $\CDR$. 

\begin{cor}\label{CorMainC1}
Let $\Delta \colon \PSsp \to \PSsp$ be a surjective isometry. 
Then there exists an isometric order isomorphism $\mathcal{T} \colon \CDR \to \CDR$ such that $\mathcal{T}|_{\PSsp} = \Delta$. 
\end{cor}

\begin{proof}
By Theorem~\ref{ThmMainC1}, every surjective isometry $\Delta \colon \PSsp \to \PSsp$ extends to a surjective complex-linear isometry $T \colon \CD \to \CD$ of the form 
\begin{align*}
T(f)(t) = f(0) + \int_0^t f'(\tau(x)) dx, 
\end{align*}
where $\tau \colon [0, 1] \to [0, 1]$ is a homeomorphism. 
The above representation of $T$ implies that $f \in \CDR$ if and only if $T(f) \in \CDR$. 
Thus the restriction $\mathcal{T} = T|_{\CDR} \colon \CDR \to \CDR$ is a surjective real-linear isometry. 

We show that $\mathcal{T}$ is an order isomorphism. 
Note first that, for any $f \in \CDR$, we see that $\mathcal{T}(f)(0) = f(0)$ and $\mathcal{T}(f)' = f' \circ \tau$. 
Fix $f, g \in \CDR$. 
Since $\tau \colon [0, 1] \to [0, 1]$ is surjective, it follows that $\mathcal{T}(f)' \le \mathcal{T}(g)'$ if and only if $f' \le g'$. 
Hence 
\begin{align*}
\mathcal{T}(f)(0) \le \mathcal{T}(g)(0) \,\,\,\, \text{and} \,\,\,\, \mathcal{T}(f)' \le \mathcal{T}(g)' \,\, \text{on} \,\, [0, 1] 
\quad \eq \quad 
f(0) \le g(0) \,\,\,\, \text{and} \,\,\,\, f' \le g' \,\, \text{on} \,\, [0, 1], 
\end{align*}
which shows that $\mathcal{T}(f) \le \mathcal{T}(g)$ if and only if $f \le g$. 
Therefore, we conclude that $\mathcal{T} \colon \CDR \to \CDR$ is an isometric order isomorphism. 
Since $\mathcal{T} = T|_{\CDR}$ and $T|_{\PSsp} = \Delta$, we also have $\mathcal{T}|_{\PSsp}=\Delta$. 
The proof is complete. 
\end{proof}

The following result is the counterpart of Corollary~\ref{CorMainC1} for $\LipR$. 

\begin{cor}\label{CorMainLip}
Let $\Delta \colon \PSLsp \to \PSLsp$ be a surjective isometry. 
Then there exists an isometric order isomorphism $\mathcal{T} \colon \LipR \to \LipR$ such that $\mathcal{T}|_{\PSLsp} = \Delta$. 
\end{cor}

\begin{proof}
The proof is analogous to that of Corollary~\ref{CorMainC1}.
We only have to replace the composition operator $f' \mapsto f' \circ \tau$ by the isometric $*$-isomorphism $\Lambda \colon L^\infty[0,1] \to L^\infty[0,1]$ appearing in Theorem~\ref{ThmMainLip}. 
Since every $*$-isomorphism preserves the order on self-adjoint elements, the resulting extension is an isometric order isomorphism on $\LipR$. 
\end{proof}


\subsection{The Space $C^n[0, 1]$ and $\AC$}

A natural question is whether analogous results can be obtained for the space $C^n[0, 1]$ of $n$-times continuously differentiable functions. 
For each $f \in C^n[0, 1]$, define
\begin{align*}
\sigp{f} = 
\begin{dcases}
\sqrt[p]{\sum_{k = 0}^{n - 1}|f^{(k)}(0)|^p + \sn{f^{(n)}}^p} & (1 \le p < \infty), \\
\max\set{|f(0)|, \dots, |f^{(n-1)}(0)|, \sn{f^{(n)}}} & (p = \infty). 
\end{dcases}
\end{align*}
Then $\sigp{\cdot}$ defines a complete norm on $C^n[0, 1]$. 
Moreover, if we define an order on $C^n_\R[0, 1]$ by $f \le g$ whenever 
\begin{align*}
f^{(k)}(0) \le g^{(k)}(0) \,\, (0 \le k \le n - 1) \quad \text{and} \quad f^{(n)} \le g^{(n)} \,\, \text{on} \,\, [0, 1], 
\end{align*}
then $C^n_\R[0, 1]$ is an ordered Banach space. 
This space admits a decomposition of the form 
\begin{align*}
C^n_\R[0, 1] \simeq \R \oplus C^{n - 1}_\R[0, 1]. 
\end{align*}
Nevertheless, the positive part $B^+(C^{n - 1}_\R[0, 1])$ of the unit ball of $C^{n - 1}_\R[0, 1]$ has diameter $2$. 
Consequently, the crucial assumption of Theorem~\ref{ThmMain1} fails, and the argument developed in Section~\ref{Sec2} cannot be applied directly.

Another space closely related to $\CD$ and $\Lip$ is the space $\AC$ of absolutely continuous functions. 
Every function $f \in \AC$ is differentiable almost everywhere on $[0, 1]$, and its derivative belongs to $L^1[0, 1]$. 
For each $f \in \AC$, define
\begin{align*}
\sigp{f} = 
\begin{cases}
\sqrt[p]{|f(0)|^p + \norm{f'}_{L^1}^p} & (1 \le p < \infty), \\
\max\set{|f(0)|, \norm{f'}_{L^1}} & (p = \infty). 
\end{cases}
\end{align*}
Equipped with this norm, $\AC$ becomes a Banach space. 
Moreover, if we define an order on $\ACR$ by $f \le g$ whenever 
\begin{align*}
f(0) \le g(0) \quad \text{and} \quad f' \le g' \,\, \text{almost everywhere on} \,\, [0, 1], 
\end{align*}
then $\ACR$ is an ordered Banach space. 
As in the cases of $\CDR$ and $\LipR$, the space $\ACR$ can be decomposed as 
\begin{align*}
\ACR \simeq \R \oplus L^1_\R[0, 1].  
\end{align*}
Unlike the spaces considered in this paper, however, the positive part of the unit ball of $L^1[0,1]$ has diameter $2$. 
Therefore, the hypothesis of Theorem~\ref{ThmMain1} is not satisfied, and the approach developed in Section~\ref{Sec2} does not apply directly to $\AC$ either. 
In addition, Lemma~\ref{LemConvexBody} is invalid in the $\AC$ case because $B^+(L^1[0, 1])$ is not a convex body, even in $L^1_\R[0, 1]$.


\section{Acknowledgments}

The fourth author (Min-Ruei Lin) is supported by the National Science and Technology Council, Taiwan (NSTC), under Grant No.~115-2917-I-110-009.


\end{document}